\documentclass{article}

\usepackage[preprint]{neurips_2021}

\usepackage[utf8]{inputenc} % allow utf-8 input
\usepackage[english]{babel}
\usepackage[T1]{fontenc}    % use 8-bit T1 fonts
\usepackage{url}            % simple URL typesetting
\usepackage{booktabs}       % professional-quality tables
\usepackage{amsfonts}       % blackboard math symbols
\usepackage{nicefrac}       % compact symbols for 1/2, etc.
\usepackage{microtype}      % microtypography
\usepackage{xcolor}         % colors
\usepackage{amsmath}
\usepackage{interval}
\usepackage{enumitem}
\usepackage{pgfplots}
\pgfplotsset{compat=1.15}
\usepackage{amsthm}
\RequirePackage{hyperref}[6.83]
\hypersetup{colorlinks,
			linkcolor=[rgb]{.61,0,0.3},
			citecolor=[rgb]{.14,.47,.14}}
			
\usepackage[capitalize,nameinlink]{cleveref}

\crefname{assumption}{Assumption}{Assumptions}
\crefname{property}{Property}{Properties}

\newtheorem{theorem}{Theorem}
\newtheorem{corollary}{Corollary}
\newtheorem{lemma}{Lemma}
\newtheorem{assumption}{Assumption}
\newtheorem{property}{Property}
\newtheorem*{remark}{Remark}

\newcommand{\E}[1]{\mathbb{E}\left[ #1 \right]}
\newcommand{\inlineE}[1]{\mathbb{E}[ #1 ]}
\newcommand{\norm}[1]{\left\Vert #1 \right\Vert}
\newcommand{\inlinenorm}[1]{\Vert #1 \Vert}
\newcommand{\Ind}[1]{\mathbf{1}\left[ #1 \right]}
\newcommand{\inlineInd}[1]{\mathbf{1}[ #1 ]}
\newcommand{\Prb}[1]{\mathbb{P}\left[ #1 \right]}
\newcommand{\inlinePrb}[1]{\mathbb{P}[ #1 ]}
\newcommand{\condE}[2]{\mathbb{E}\left[ \left. #1 \right\vert #2 \right]}
\newcommand{\inlinecondE}[2]{\mathbb{E}[ #1 | #2 ]}
\newcommand{\condPrb}[2]{\mathbb{P}\left[ \left. #1 \right\vert #2 \right]}
\newcommand{\inlinecondPrb}[2]{\mathbb{P}[  #1 \vert #2 ]}
\newcommand{\Flb}{F_{l.b.}}

\title{Stochastic Gradient Descent on Nonconvex Functions with General Noise Models}

\author{%
  Vivak Patel \\
  Department of Statistics\\
  University of Wisconsin\\
  Madison, WI 53706\\
  \texttt{vivak.patel@wisc.edu} \\
   \And
   Shushu Zhang \\
   Department of Statistics \\
   University of Wisconsin\\
   Madison, WI 53706 \\
   \texttt{szhang695@wisc.edu} \\
}
\begin{document}

\maketitle

\begin{abstract}
Stochastic Gradient Descent (SGD) is a widely deployed optimization procedure throughout data-driven and simulation-driven disciplines, which has drawn a substantial interest in understanding its global behavior across a broad class of nonconvex problems and noise models. Recent analyses of SGD have made noteworthy progress in this direction, and these analyses have innovated important and insightful new strategies for understanding SGD. However, these analyses often have imposed certain restrictions (e.g., convexity, global Lipschitz continuity, uniform H\"{o}lder continuity, expected smoothness, etc.) that leave room for innovation. In this work, we address this gap by proving that, for a rather general class of nonconvex functions and noise models, SGD's iterates either diverge to infinity or converge to a stationary point with probability one. By further restricting to globally H\"{o}lder continuous functions and the expected smoothness noise model, we prove that---regardless of whether the iterates diverge or remain finite---the norm of the gradient function evaluated at SGD's iterates converges to zero with probability one and in expectation. As a result of our work, we broaden the scope of nonconvex problems and noise models to which SGD can be applied with rigorous guarantees of its global behavior.
\end{abstract}

\section{Introduction}
Stochastic gradient descent (SGD) is widely deployed throughout data science and adjacent fields to solve
\begin{align}\label{eqn:obj}
\min_{\theta} F(\theta)
\end{align}
where $F : \mathbb{R}^p \to \mathbb{R}$ and is defined to be the expectation of a function $f : \mathbb{R}^p \times \mathcal{X} \to \mathbb{R}$, where $\mathcal{X}$ is the range of a well-defined random variable $X$.\footnote{Hence, we can assume that there is a $\sigma$-Algebra defined on sets of $\mathcal{X}$ that ensure it is a measure space that can support $X$. We can also define an appropriate push forward measure to specify a probability space. Consequently, we can define an expectation with respect to this probability space.}
Owing to its broad usage, SGD's global behavior on different classes of functions $f$ (and, hence, $F$) has been of substantial interest. While there are many works that have provided insight, understanding SGD's global behavior has been notably advanced by several recent works \citep{asi2019,lei2019,patel2020,khaled2020} that we overview presently. 

To explain the insights of these works, we will need some notation. We define $\dot{F}(\theta)$ as the gradient of $F$ evaluated at the point $\theta \in \mathbb{R}^p$, and we define $\dot{f}(\theta,X)$ as the gradient of $f$ with respect to its first argument evaluated at $(\theta,X)$, which follows the notation of \citet{patel2020}. We now describe these essential works.
\begin{enumerate}[leftmargin=*,itemsep=0pt]
\item \citet{asi2019} show that when $f(\cdot, x)$ is a closed, convex, subdifferentiable function for all $x \in \mathcal{X}$, then SGD's iterates are stable with probability one and will converge to a solution under some additional assumptions. Distinguishingly, \citet{asi2019} allow $\inlineE{ \inlinenorm{ \dot f(\theta, X)}_2^2}$ (i.e., the noise model) to grow arbitrarily with the distance between the current iterate and the solution set. To our knowledge, this is the most general assumption for the noise under which convergence has been demonstrated, and the proof relies intimately on convexity \citep[see][Lemma 3.7]{asi2019}.
\item \citet{lei2019} prove that for uniformly lower bounded, nonconvex functions, $f$, for which 
\begin{equation} \label{eqn-uniform-holder}
(\exists L > 0 )(\forall \theta_1, \theta_2 \in \mathbb{R}^p)(\forall x \in \mathcal{X}) : \norm{ \dot{f}(\theta_1, x) - \dot{f}(\theta_2, x)}_2 \leq L \norm{ \theta_1 - \theta_2 }_2^\alpha,
\end{equation}
with $\alpha \in (0,1]$, the objective function, $F$, evaluated at the SGD iterates converges almost surely to a bounded random variable. Moreover, \citet{lei2019} show that, when $\alpha = 1$, the expected value of the norm of the gradient function, $\dot F$, evaluated at the SGD iterates converges to zero.
\item \citet{patel2020} shows that for a lower bounded, nonconvex objective function, $F$, for which $\dot{F}$ is globally Lipschitz continuous and for which
\begin{equation} \label{eqn-noise-control}
(\exists C_1, C_2 \geq 0)(\forall \theta \in \mathbb{R}^p):\E{ \norm{\dot f(\theta,X) }_2^2} \leq C_1 + C_2 \norm{ \dot{F}(\theta) }_2^2,
\end{equation}
the norm of the gradient function, $\dot F$, evaluated at the SGD iterates converges to zero with probability one. Moreover, \citet{patel2020} allows for matrix-valued learning rates. The later global convergence work of \citet{mertikopoulos2020} offers similar conclusions under more stringent conditions, but also explores local properties such as local rates of convergence.
\item \citet{khaled2020} show that for a lower bounded, nonconvex objective function, $F$, for which $\dot{F}$ is globally Lipschitz continuous and for which
\begin{equation} \label{eqn-noise-expected-smoothness}
(\exists C_1, C_2, C_3 \geq 0)(\forall \theta \in \mathbb{R}^p):\E{ \norm{\dot f(\theta,X)}_2^2} \leq C_1 + C_2 \norm{ \dot{F}(\theta)}_2^2 + C_3 F(\theta),
\end{equation}
the \textit{smallest} of all expected norms of the gradient evaluated at the SGD iterates converges to zero. Similar results are explored by \citet{gower2020}.
\end{enumerate}

\subsection*{Contributions}
Our goal here is to move towards a more general theory of convergence that combines all of these threads under a single analysis framework. Specifically, by innovating on the strategies of \cite{lei2019} and \cite{patel2020}, we will prove the following results for SGD with matrix-valued learning rates, which we state informally now and formalize later.
\begin{enumerate}[leftmargin=*,itemsep=0pt]
\item We prove that for a lower bounded, nonconvex objective function, $F$, for which $\dot{F}$ is locally $\alpha$-H\"{o}lder continuous and for which $\inlineE{ \inlinenorm{ \dot f(\theta,X) }_2^2}$ is controlled by an arbitrary, non-negative upper semi-continuous function, either the iterates of SGD diverge to infinity, or they remain finite. See \cref{theorem-capture,theorem-convergence-local-holder-general-noise} in \cref{section-capture}. 
\item When the iterates remain finite, the objective function, $F$, evaluated at the iterates converges to a finite random variable, and  the gradient norm evaluated at the iterates converges to zero with probability one. With this result, we are able to relax the noise models of \citet{bottou2018, asi2019, lei2019, patel2020, khaled2020}; relax the global, uniform H\"{o}lder continuity assumption of \citep{lei2019}; and relax the global, Lipschitz continuity assumption of \citet{bottou2018,patel2020,khaled2020}. See \cref{theorem-convergence-local-holder-general-noise} in \cref{section-capture}.
\item When the iterates diverge, we can also say something interesting under slightly stronger conditions. Specifically, by strengthening the local H\"{o}lder assumption to a global H\"{o}lder assumption of $\dot{F}$ and restricting the noise model on $\dot{f}$ to \eqref{eqn-noise-expected-smoothness}, we are able to show that, regardless of the iterate behavior, the objective function evaluated at the iterates converges with probability one to an \textit{integrable} random variable, and the norm of the gradient function evaluated at the iterates converges to zero with probability one \textit{and} in $L^1$. This result directly generalize the results of \citet{bottou2018,lei2019,patel2020,khaled2020}, and a host of other more specialized results that are covered by these works. See \cref{theorem-global-holder-expected-smoothness} in \cref{section-global-holder-expected-smoothness}.
\end{enumerate}

To our knowledge, our results are the most general for the global analysis of SGD as they allow for rather general nonconvex functions (e.g., locally H\"{o}lder gradient function) and general noise models (e.g., arbitrary, upper semi-continuous bound on the second moment). As a result, our results broaden the scope of problems to which SGD can be used with rigorous guarantees of its asymptotic behavior.

\subsection*{Organization}

In \cref{section-sgd}, Stochastic Gradient Descent (SGD) with matrix-valued learning rates are precisely specified. In \cref{section-capture}, SGD's iterates are shown to either diverge or remain finite over a general class of nonconvex functions and noise models, which are precisely specified in this section; moreover, when SGD's iterates remain finite, then they are shown to converge to a stationary point with probability one. In \cref{section-global-holder-expected-smoothness}, the gradient function evaluated at SGD's iterates is shown to converge to zero with probability one and in $L^1$, under the stronger assumptions of global H\"{o}lder continuity and under a more restricted noise model. In \cref{section-conclusion}, we conclude this work with a discussion of limitations and future work.

\section{Stochastic Gradient Descent} \label{section-sgd}
We define Stochastic Gradient Descent to be the procedure that beings with an arbitrary $\theta_0 \in \mathbb{R}^p$ and generates $\lbrace \theta_{k} : k \in \mathbb{N} \rbrace$ according to the recursion
\begin{align}\label{eqn:iterates}
\theta_{k+1}=\theta_{k}-M_k \dot{f}(\theta_k,X_{k+1}),
\end{align} 
where $\lbrace X_k : k \in \mathbb{N} \rbrace$ are independent and are identically distributed to $X$; and $\lbrace M_k : k+1 \in \mathbb{N} \rbrace \subset \mathbb{R}^{p\times p}$ are matrices whose properties we specify momentarily. Let $\mathcal{F}_0 = \sigma(\theta_0)$ and $\mathcal{F}_k = \sigma(\theta_0, X_1,\ldots,X_k)$ for all $k \in \mathbb{N}$. 
\begin{remark}
We note that if $\theta_0$ is random, then we will condition the results below on $\mathcal{F}_0 = \sigma(\theta_0)$. However, to avoid this additional notation, we will not state this explicitly.
\end{remark}

First, we will require that 
\begin{property}\label{P1}
	$\lbrace M_k : k + 1 \in \mathbb{N} \rbrace$ are symmetric, positive definite matrices.
\end{property}
\cref{P1} is a natural extension to the scalar learning rate case in which the learning rate is required to be positive valued at each iterate.

Second, we consider the natural extension of the Robbins-Monro condition for $\alpha$-H\"{o}lder continuous functions with matrix-valued learning rates. Specifically, we require that $\lbrace M_k \rbrace$ satisfy
\begin{property}\label{P2}
	$\sum_{k=1}^{\infty} \lambda_{\max}(M_k)^{1+\alpha} =: S < \infty$ with $\alpha \in (0,1]$,
\end{property}
and
\begin{property}\label{P3}
	$\sum_{k=1}^{\infty} \lambda_{\min}(M_k) = \infty$,
\end{property}
where $\lambda_{\max}(\cdot)$ and $\lambda_{\min}(\cdot)$ denote the largest and smallest eigenvalues of the given symmetric matrix.

For \cref{section-global-holder-expected-smoothness}, we will require the following condition, which controls the relationship between $\lambda_{\max}(M_k)$ and $\lambda_{\min}(M_k)$. Note, such a condition is readily satisfied for scalar learning rates satisfying \cref{P2}.
\begin{property}\label{P4}
	$\lim_{n \to \infty} \lambda_{\max}(M_k)^\alpha \kappa(M_k)=0$ where $\kappa(M_k) = \norm{M_k}_2 \norm{M_k^{-1}}_2$.
\end{property}

%Finally, for \cref{section-radial-growth}, we will require the following strengthening of \cref{P2}.
%\begin{property} \label{P5}
%$\exists \epsilon > 0$ such that
%	$\sum_{k=1}^\infty \lambda_{\max}(M_k)^{1 + \alpha - \epsilon} =: S_\epsilon < \infty$.
%\end{property}

\section{A Capture Theorem and Its Consequences} \label{section-capture}
We will begin by defining a set of assumptions under which we will analyze SGD and we will discuss how it relates to the assumptions in the aforementioned works. We refer the reader to \S 2 of \citet{patel2020} for a review of common assumptions in the nonconvex landscape and their relationships.

We begin with a common assumption that ensure that minimizing the objective function is a reasonable effort. 
\begin{assumption} \label{A1}
There exists $\Flb \in \mathbb{R}$ such that $\Flb \leq F(\theta)$ for all $\theta \in \mathbb{R}^p$.
\end{assumption}
Indeed, this is the assumption made by \citet{khaled2020} and \citet{patel2020}. This assumption is implied by \citet{lei2019}'s more stringent assumption that $f(\theta,x) \geq 0$ for all $\theta \in \mathbb{R}^p$ and for all $x \in \mathcal{X}$. Finally, this assumption is implied by \citet{asi2019}'s assumptions that $F$ is convex and that the optimization problem has a nonempty solution set.

We also require the common assumption that the stochastic gradient $\dot{f}(\theta,X)$ are unbiased. We note that this assumption is common in most works, and can be relaxed as shown in \S 4 of \citet{bottou2018}. Fortunately, this relaxation is rather easy to account for within our analysis.
\begin{assumption} \label{A2}
For all $\theta \in \mathbb{R}^p$, $\inlineE{ \dot f(\theta,X)} = \dot{F}(\theta)$. 
\end{assumption}

We now come to our two less common, yet more general assumptions in comparison to what can be found in the literature. The first assumption is inspired by the noise model assumption of \citet{asi2019}, which allows the trace of the variance of $\dot{f}(\theta,X)$ to grow with $\theta$'s distance from the assumed solution set. Here, we have no such luxury of having a guaranteed solution set, and so we require a more general assumption.
\begin{assumption} \label{A3a}
Let $G : \mathbb{R}^p \to \mathbb{R}_{\geq 0}$ be an upper semi-continuous function. For all $\theta \in \mathbb{R}^p$, $\inlineE{ \inlinenorm{ \dot f (\theta, X) }_2^2} \leq G(\theta)$. 
\end{assumption}
We see that \cref{A3a} readily generalized the noise modeling assumptions of \citet{khaled2020} and \citet{patel2020}. Moreover, \cref{A3a} is implied under the more stringent conditions in \citet{lei2019}, specifically by using \eqref{eqn-uniform-holder} and \cref{A1} with an application of Young's inequality.

For the last assumption, we recall that a function is locally $\alpha$-H\"{o}lder continuous if for any compact set $\mathcal{K} \subset \mathbb{R}^p$, $\exists L > 0$ such that $\forall \varphi_1, \varphi_2 \in \mathcal{K}$, $\inlinenorm{ \dot F (\varphi_1) - \dot F (\varphi_2) }_2 \leq L \inlinenorm{ \varphi_1 - \varphi_2}_2^\alpha$. 
\begin{assumption} \label{A4a}
$\dot{F}$ is locally $\alpha$-H\"{o}lder continuous for some $\alpha \in (0,1]$.
\end{assumption}
Again, \cref{A4a} is weaker than the global Lipschitz assumptions of \citet{khaled2020} and \citet{patel2020}, and is implied by the global, uniform $\alpha$-H\"{o}lder continuity assumed in \citet{lei2019}. Interestingly, \cref{A4a} is cleverly circumvented in \citet[Lemma 3.6]{asi2019} using the monotonicity of the gradient function and Young's inequality, and one could argue that it would generalize \cref{A4a}. However, this argument would fall apart for matrix-valued learning rates, as the monotonicity of the gradient operator is no longer guaranteed even in the convex case. 

With these assumption, we begin by defining a central property of SGD that is often overlooked or implicitly required, and has several immediate consequences and applications. The proof is a direct application of the Borel-Cantelli lemma and can be found in \cref{section-analysis-general}.
\begin{theorem}[Capture Theorem] \label{theorem-capture}
Let $\bar \theta \in \mathbb{R}^p$ be arbitrary. Let $\lbrace \theta_k \rbrace$ be defined as in \eqref{eqn:iterates} and satisfy \cref{P1,P2}. If \cref{A3a} holds, then for any $R \geq 0$, 
\begin{equation}
\Prb{ \norm{ \theta_{k+1} - \bar \theta}_2 > R, ~ \norm{ \theta_k - \bar \theta}_2 \leq R~i.o.} = 0.
\end{equation}
\end{theorem}

\cref{theorem-capture} has several immediate consequences. For example, \cref{theorem-capture} is central to proving local rates of convergence as it ensures that the iterates are eventually captured within some basin of attraction, in which some sort of local analysis can be done; however, such a local analysis is not the focus of this work. For a global perspective, \cref{theorem-capture} implies the following result.

\begin{theorem} \label{theorem-convergence-local-holder-general-noise}
Let $\lbrace \theta_k \rbrace$ be defined as in \eqref{eqn:iterates} and satisfy \cref{P1,P2,P3}. Suppose \cref{A1,A2,A3a,A4a} hold. Let $\mathcal{A}_1 = \lbrace \liminf_{k \to \infty} \norm{ \theta_k}_2 = \infty \rbrace$ and $\mathcal{A}_2 = \lbrace \lim_{k \to \infty} \norm{ \theta_k}_2 < \infty \rbrace$. Then, the following statements hold
\begin{enumerate}[leftmargin=*,itemsep=0em]
\item $\inlinePrb{ \mathcal{A}_1} + \inlinePrb{ \mathcal{A}_2} = 1$.
\item There exists a finite random variable, $F_{\lim}$, such that, on $\mathcal{A}_2$, $\lim_{k \to \infty} F(\theta_k) = F_{\lim}$ and $\lim_{k \to \infty} \inlinenorm{ \dot{F}(\theta_k) }_2 = 0$ with probability one.
\end{enumerate} 
\end{theorem}
\begin{proof}[Proof of \cref{theorem-convergence-local-holder-general-noise}]
Note, the referenced results can be found in \cref{section-analysis-general}. Informally, \cref{theorem-capture} implies that the limit supremum and limit infimum of $\inlinenorm{\theta_k}_2$ cannot be distinct. As a result, the limit of $\inlinenorm{\theta_k}$ is either infinite or is finite with probability one. This is formalized in \cref{corollary-limits}.

To show the remaining statement, we begin by constraining to the event $\lbrace \sup_{k} \inlinenorm{\theta_k}_2 \leq R \rbrace$ for arbitrary $R \geq 0$. Then, on this event we can prove that $F(\theta_k)$ converges to a finite random variable. Similarly, we can prove that on this event $\inlinenorm{ \dot{F}(\theta_k)}_2$ must converge to zero. By taking the union over all $R \in \mathbb{N}$, we can conclude that these two statements hold on the event $\lbrace \sup_k \norm{\theta_k}_2 < \infty \rbrace$ which is implied by $\mathcal{A}_2$. These arguments are formalized in \cref{corollary-convergence-objective-constrained} for the objective function statement, and \cref{corollary-convergence-grad-constrained} for the gradient function statement. For the gradient function statement, the proof strategy is adapted from \citet{patel2020}.
\end{proof}

\begin{remark}
By \cref{theorem-capture}, we have that $\lbrace \lim_{k \to \infty} \inlinenorm{ \theta_k}_2 < \infty \rbrace$ is equal to $\lbrace \lim_{k \to \infty} \inlinenorm{ \theta_k - \bar \theta }_2 < \infty \rbrace$ for an arbitrary choice of $\bar{\theta}$. Thus, by choosing $\bar{\theta}_1$ and $\bar{\theta}_2$ such that $\lbrace 0, \bar \theta_1, \bar \theta_2 \rbrace$ are not colinear, then we conclude by \cref{theorem-capture} and triangulation that $\theta_k$ converges to a finite random variable on $\mathcal{A}_2$. Hence, by \cref{theorem-convergence-local-holder-general-noise}, $\lbrace \theta_k \rbrace$ converges to a finite random variable that takes value over the stationary points of the objective function on $\mathcal{A}_2$. Or, to be more succinct, we will say that $\lbrace \theta_k \rbrace$ converges to a stationary point on $\mathcal{A}_2$.
\end{remark}

We note that the statement of $\Prb{ \mathcal{A}_1} + \Prb{\mathcal{A}_2} = 1$ is not trivial. There is no \textit{apriori} guarantee that the limit supremum and limit infimum of $\inlinenorm{\theta_k}_2$ must coincide (cf., a simple random walk with a positive one bias at each step, which will have its limit supremum as infinity and limit infimum as zero with probability one). Moreover, in the case that the limit exists, it is also nontrivial that the procedure converges to a stationary point.

One case that is not explored by \cref{theorem-convergence-local-holder-general-noise} is the case of $\mathcal{A}_1$. There are two possibilities here: either we want to allow $\norm{\theta_k}_2$ to diverge as the gradient function can be zero in the limit, or we want to disallow the possibility of $\mathcal{A}_2$ entirely (i.e., $\Prb{ \mathcal{A}_2} = 0$). This first possibility is the focus of \cref{section-global-holder-expected-smoothness}, while the second possibility will not be explored in this work.

\section{Global H\"{o}lder Continuity and Expected Smoothness} \label{section-global-holder-expected-smoothness}
Here, we consider the situation in which the iterates diverging (i.e., event $\mathcal{A}_1$ in \cref{theorem-convergence-local-holder-general-noise}) might be meaningful to the underlying optimization problem. As a simple example of a situation where this may occur, consider optimizing the smooth rectifier function as described in \cref{figure-smooth-rectifier}. In this example, the optimizer would find $\lbrace \theta_k \rbrace$ with $\theta \to -\infty$. 

\begin{figure}[hb]
\centering
\begin{tikzpicture}[scale=0.7]
\begin{axis}[axis lines = left,
    xlabel = $\theta$,
    ylabel = {$F(\theta)$}]
    \addplot[domain=-10:10, black, ultra thick] {ln( 1 + exp(x))};
\end{axis}
\end{tikzpicture}
\caption{A plot of the smooth rectifier function $F(\theta) = \log( 1 + \exp( \theta) )$, where $\theta \in \mathbb{R}$.}
\label{figure-smooth-rectifier}
\end{figure}
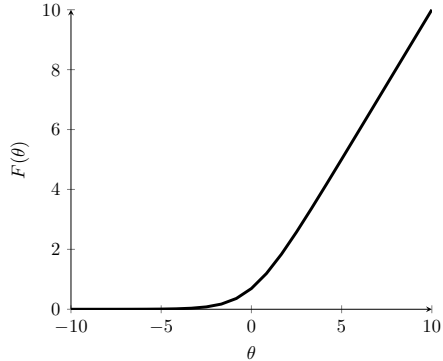

To address this case, we will need to strength \cref{A3a,A4a}. We will begin by strengthening \cref{A3a} with the following assumption, termed \textit{expected smoothness} \citep{khaled2020}.
\begin{assumption} \label{A3b}
There exists $C_1, C_2 \geq 0$ and $C_3 \geq 1$ such that, $\forall \theta \in \mathbb{R}^p$,
\begin{equation}
\E{ \norm{ \dot{f}(\theta,X)}_2^2} \leq C_1 + C_2 ( F(\theta) - \Flb) + C_3 \norm{ \dot F (\theta) }_2^2.
\end{equation}
\end{assumption}
We note that \cref{A3b} is implicitly making use of \cref{A1} owing to the term $\Flb$, but this can be easily addressed by removing $\Flb$ and ensuring that the upper bound is non-negative (e.g., by using $\max \lbrace F(\theta),0 \rbrace$ in place of $F(\theta) - \Flb$). We also note that the requirement $C_3 \geq 1$, in conjunction with \cref{A2}, implies
\begin{equation}
\E{ \norm{ \dot{f}(\theta,X) - \dot F(\theta)}_2^2 } \leq C_1 + C_2 ( F(\theta) - \Flb) + (C_3 - 1) \norm{ \dot F(\theta)}_2^2,
\end{equation}
which allows for the variance to be well-specified (i.e., the upper bound is non-negative).

We now turn our attention to strengthening \cref{A4a} as follows.
\begin{assumption} \label{A4b}
$\dot{F}$ is globally $\alpha$-H\"{o}lder continuous for some $\alpha \in (0,1]$.
\end{assumption}

While \cref{A4b} is more restrictive than \cref{A4a} and precludes certain objective functions, it is rather natural for the case in which $\inlinenorm{\theta_k} \to \infty$ may be meaningful for the optimization problem. Note, by \cref{lemma-grad-bounded-objective}, \cref{A1,A3b,A4b} together imply that $\exists C_4 \geq 0$ such that
\begin{equation}
\E{ \norm{ \dot{f}(\theta,X)}_2^2} \leq C_1 + C_2 ( F(\theta) - \Flb) + C_4 (F(\theta) - \Flb)^{\frac{2\alpha}{1+\alpha}}.
\end{equation}
In the upper bound, we note that the the third term can be absorbed into the first and second term by separating the cases of $F(\theta) - \Flb \leq 1$ and $F(\theta) - \Flb > 1$, which then implies that \cref{A3b} holds with different choices of constants. Thus, even in light of \cref{A1,A4b}, \cref{A3b} is quite general.

Now, under \cref{A1,A2,A3b,A4b}, an SGD procedure satisfying \cref{P1,P2,P3,P4} obeys the following result, which states that the objective function evaluated at the iterates converges to an integrable random variable, and the norm of the gradient function evaluated at the iterates converges to zero with probability one and in $L^1$. As a result, even if the iterates diverge, we see that they are still tending to regions where the gradient is zero. This result generalizes the global analysis results of \citet{bottou2018}, \citet{lei2019}, \citet{patel2020} and \citet{khaled2020}.
\begin{theorem} \label{theorem-global-holder-expected-smoothness}
Let $\lbrace \theta_k \rbrace$ be defined as in \eqref{eqn:iterates} and satisfy \cref{P1,P2,P3,P4}. 
Suppose \cref{A1,A2,A3b,A4b} hold. 
Then, 
\begin{enumerate}[leftmargin=*,itemsep=0em]
\item there exists an integrable random variable, $F_{\lim}$, (i.e., $\E{ F_{\lim} } < \infty$) such that $\lim_{k \to \infty} F(\theta_k) = F_{\lim}$ with probability one.
\item Moreover, $\sup_{k+1 \in \mathbb{N} } \inlineE{ F(\theta_k) } < \infty$. Therefore, $\forall \gamma \in [0,1)$, $$\lim_{k \to \infty} \E{ |(F(\theta_k)-\Flb)^\gamma - (F_{\lim}-\Flb)^\gamma| } = 0.$$
\item Finally, $\lim_{k \to \infty} \inlinenorm{ \dot F (\theta_k)}_2 = 0$ with probability one and $\lim_{k \to \infty} \inlineE{ \inlinenorm{ \dot F(\theta_k)}_2 } = 0$. 
\end{enumerate} 
\end{theorem}
\begin{proof}
Note, the reference results can be found in \cref{section-analysis-infinite}.
The proofs follow a rather similar strategy to that of \cref{theorem-convergence-local-holder-general-noise} for demonstrating convergence of $\lbrace F(\theta_k) \rbrace$ and $\lbrace \inlinenorm{ \dot F(\theta_k) }_2 \rbrace$ with probability one. The details are supplied in \cref{corollary-convergence-objective-unconstrained,corollary-convergence-grad-unconstrained}, respectively.

The proof for bounding $\sup_{k+1 \in \mathbb{N} } \inlineE{ F(\theta_k) }$ follows by setting up a recursive relationship between sequential objective function values and making use of \cref{P2} to show that they are uniformly bounded. The strategy is adapted from \citet{lei2019}. 
Once we can bound the supremum, we have that $\lbrace (F(\theta_k)-\Flb)^\gamma \rbrace$ are uniformly integrable for $\gamma \in [0,1)$. By combining this observation with strong convergence, we have that $\lbrace (F(\theta_k)-\Flb)^\gamma \rbrace$ converges to $(F_{\lim}-\Flb)^\gamma$ in $L^1$. 
The details are supplied in \cref{corollary-convergence-expected-objective-unconstrained}. 

Once the objective function is controlled, the supremum over the squared norm of the gradient function at the iterates can be controlled using well-known inequalities that make use of \cref{A1}, \cref{A4b} and the fundamental theorem of calculus. As a result, we have that the squared norm of the gradient function at the iterates is uniformly integrable. By combining this with strong convergence, we have that the norm of the gradient function evaluated at the iterates converges to zero in $L^1$. The details are supplied in \cref{corollary-convergence-expected-grad-unconstrained}.
\end{proof}

\begin{remark}
If it is of interest, under the setting of \cref{theorem-global-holder-expected-smoothness}, we can preclude divergence by requiring a radially coercive objective function (i.e., that the objective function tends to infinity as the norm of its argument goes to infinity). To be specific, under this additional requirement, the event $\mathcal{A}_1$ from \cref{theorem-convergence-local-holder-general-noise} coincides with the event on which $(F(\theta_k)-\Flb)^{1/2}$ diverges to infinity; however, by Markov's inequality, the probability that $ (F (\theta_k) - \Flb)^{1/2} $ exceeds some value $\ell$ is bounded by the product of $\ell^{-1}$ and the $\sup_j \inlineE{ ( F(\theta_j) - \Flb )^{1/2} }$, where this final quantity is finite by \cref{theorem-global-holder-expected-smoothness}. Hence, $\inlinePrb{\mathcal{A}_1}$ can be made arbitrarily small, which implies that it is zero.
\end{remark}

\section{Conclusion} \label{section-conclusion}
In this work, we have, to our knowledge, provided results about the global behavior of SGD on the most general nonconvex functions and noise models currently in the literature. In particular, we prove that SGD's iterates either diverge or remain finite, and, in the latter case, we prove that SGD's iterates converge to a stationary point with probability one. Moreover, if we restrict the class of nonconvex functions to those that are globally H\"{o}lder continuous and the noise model to expected smoothness, then we prove that $\dot{F}$ evaluated at SGD's iterates converges to zero with probability one and in $L^1$. With these results, we broaden the scope of problems to which SGD can be applied with rigorous guarantees of its asymptotic behavior.

\subsection*{Limitations}  

A key limitation of this work is that, under the general setting of local H\"{o}lder continuity and generic noise model (i.e, \cref{theorem-convergence-local-holder-general-noise}), we have not provided insight into what happens on the event that the iterates diverge. That is, in \cref{theorem-convergence-local-holder-general-noise}, we do not say anything about $\mathcal{A}_1$ (the event on which the iterates diverge), except that it exists. Ideally, we would like to show that such an even has probability zero, but this is not necessarily true \textit{even for objective functions that are radially coercive, so long as the noise model grows sufficiently rapidly}.

\subsection*{Future Work}

Our primary goal for future work is to address the above limitations. Specifically, our goal is to determine general, reasonable conditions on $F$ that will imply that $\inlinePrb{ \liminf_{k \to \infty} \inlinenorm{ \theta_k }_2 = \infty } = 0$. Once this limitation is address, we will look to integrate current local analyses (i.e., local rates of convergence) with the general class of nonconvex functions and noise models established in this work. 

%\section{Radial Growth Conditions and Their Consequences} \label{section-radial-growth}
%\input{sections/radial_growth}
\pagebreak 
\bibliographystyle{abbrvnat}
{\small

\bibliography{nonconv_sgd}}

\pagebreak 
\appendix
\section{Technical Lemmas}
The first lemma is a straightforward conclusion of \cref{P1,P4}.

\begin{lemma}\label{Lemma4}
Suppose $\lbrace M_k : k +1 \in \mathbb{N} \rbrace$ satisfy \cref{P1,P4}, then $\forall C > 0$,
	$\exists K \in \mathbb{N}$ such that $\forall k \geq K$,
	\begin{align} \label{eqn-lemma-eigenvaluelb}
	\lambda_{\min}(M_k)-\frac{C}{2} \lambda_{\max}(M_k)^{1+\alpha} \geq \frac{1}{2} \lambda_{\min}(M_k).
	\end{align}
\end{lemma}
\begin{proof}[Proof of \cref{Lemma4}] Fix $C > 0$. By rearranging \eqref{eqn-lemma-eigenvaluelb}, it is equivalent to prove that $\exists K \in \mathbb{N}$ such that for all $k \geq K$, $1/C \geq  \lambda_{\max}(M_k)^{\alpha} \kappa(M_k)$. This follows directly by \cref{P4}.
\end{proof}

Let $B(R) \subset \mathbb{R}^p$ denote the open ball around zero with radius $R \geq 0$. Let $\overline{ B(R) }$ denote the closure of said ball. The following lemma is a standard consequence of the fundamental theorem of calculus and the continuity assumptions on $\dot{F}$. 

\begin{lemma}\label{lemma-ftc}
Suppose \cref{A4a} holds. Then, for any $R > 0$, $\exists L_{R} > 0$ such that $\forall \theta, \varphi \in \overline{B(R)}$ and $\forall \tilde{L} \geq L_R$,
	\begin{align}
	F(\theta) \leq F(\varphi) + \dot{F}(\varphi)'(\theta-\varphi) + \frac{\tilde{L}}{1+\alpha} \norm{ \theta - \varphi }_2^{1+\alpha}
	\end{align}
Suppose \cref{A4b} holds. Then, $\forall \theta, \varphi \in \mathbb{R}^p$, $\exists L > 0$, $\forall \tilde{L} \geq L$,
\begin{align}
	F(\theta) \leq F(\varphi) + \dot{F}(\varphi)'(\theta-\varphi) + \frac{\tilde{L}}{1+\alpha} \norm{ \theta - \varphi }_2^{1+\alpha}.
	\end{align}
\end{lemma}

\begin{proof}
By fundamental theorem of calculus and \cref{A4a}, $\exists L_{R} > 0$ such that $\forall \theta, \varphi \in \overline{B(R)}$,
\begin{align}
F(\theta) &= F(\varphi) + \dot F(\varphi)'(\theta - \varphi) + \int_0^1 \left[ \dot F (\varphi + t(\theta - \varphi) ) - \dot F (\varphi) \right]'(\theta - \varphi) dt \\
         &\leq F(\varphi) + \dot F(\varphi)'(\theta - \varphi) + \int_0^1 \norm{ \dot F (\varphi + t(\theta - \varphi) ) - \dot F (\varphi)}_2 \norm{\theta - \varphi}_2 dt \\
		 &\leq F(\varphi) + \dot F(\varphi)'(\theta - \varphi) + L_{R} \norm{ \theta - \varphi}_2^{1 + \alpha} \int_0^1 t^\alpha dt.
\end{align}
Computing the integral gives the first result. The case for \cref{A4b} is proved nearly identically.
\end{proof}

The following lemma allows us to relate smaller moments of the norm of the stochastic gradients, $\inlinenorm{ \dot{f}(\theta,X)}_2$, to the second moment.

\begin{lemma}\label{lemma-variance-control}
Let $\alpha \in (0,1]$. Let $\mathcal{F}$ be a $\sigma$-algebra. Then, for all $\theta \in \mathbb{R}^p$,
\begin{align}
\condE{ \norm { \dot{f}(\theta,X)}_2^{1+\alpha}}{\mathcal{F}} \leq
\condE{  \norm { \dot{f}(\theta,X)}_2^{2}}{\mathcal{F}}^{\frac{1+\alpha}{2} } 
\leq \left(\frac{1+\alpha}{2}\right)\condE{ \norm { \dot{f}(\theta,X)}_2^2}{\mathcal{F}}+\frac{1-\alpha}{2}
\end{align}
\end{lemma}

\begin{proof}
	If $\alpha=1$, the result holds.
	Suppose $\alpha\in (0,1)$. Then $1+\alpha < 2$ and H\"{o}lder's inequality implies
	\begin{equation}
	\condE{ \norm { \dot{f}(\theta,X)}_2^{1+\alpha}}{\mathcal{F}} \leq \condE{  \norm { \dot{f}(\theta,X)}_2^{2}}{\mathcal{F}}^{\frac{1+\alpha}{2} }.
	\end{equation}
	Now, we recall Young's inequality: $uv=\frac{u^p}{p}+\frac{v^q}{q}$ where $p>1$ and $\frac{1}{p}+\frac{1}{q}=1$. Let $u=\inlinecondE{  \inlinenorm { \dot{f}(\theta,X)}_2^{2}}{\mathcal{F}}^{\frac{1+\alpha}{2} }$, $v=1$, $p = \frac{2}{1+\alpha}$ and $q=\frac{2}{1-\alpha}$. Then, Young's inequality completes the result.
\end{proof}

\section{Analysis of the Local H\"{o}lder Continuity and General Noise Model Case} \label{section-analysis-general}
We will begin with a proof of \cref{theorem-capture}. Then, we split the proof of \cref{theorem-convergence-local-holder-general-noise} into three pieces. The first piece deals with specifying $\mathcal{A}_1$ and $\mathcal{A}_2$, and showing the sum of their probabilities is one. The second piece analyzes the objective function behavior of the iterates on $\mathcal{A}_2$. The third piece analyzes the gradient function behavior of the iterates on $\mathcal{A}_2$. 

\subsection{The Capture Theorem}
Recall that \cref{theorem-capture} states the following. Let $\bar \theta \in \mathbb{R}^p$ be arbitrary. Let $\lbrace \theta_k \rbrace$ be defined as in \eqref{eqn:iterates} and satisfy \cref{P1,P2}. If \cref{A3a} holds, then for any $R \geq 0$, 
\begin{equation}
\Prb{ \norm{ \theta_{k+1} - \bar \theta}_2 > R, ~ \norm{ \theta_k - \bar \theta}_2 \leq R~i.o.} = 0.
\end{equation}

\begin{proof}[Proof of \cref{theorem-capture}]
Fix $ R \geq 0 $ and let $\epsilon > 0$. Then,
\begin{align}
&\Prb{ \norm{ \theta_{k+1} - \bar \theta}_2 \geq R + \epsilon, \norm{ \theta_k - \bar \theta}_2 \leq R } \\
&\Prb{ \norm{ \theta_{k+1} - \bar \theta}_2 \Ind{ \norm{ \theta_k - \bar \theta}_2 \leq R } \geq R + \epsilon } \\
&\quad = \Prb{ \left( \norm{ \theta_{k+1} - \bar \theta}_2 - \norm{ \theta_k - \bar \theta}_2  + \norm{ \theta_k - \bar \theta}_2  \right) \Ind{ \norm{ \theta_k - \bar \theta}_2 \leq R } \geq R + \epsilon } \\
&\quad \leq \Prb{ \left( \norm{ \theta_{k+1} - \bar \theta}_2 - \norm{ \theta_k - \bar \theta}_2  \right) \Ind{ \norm{ \theta_k - \bar \theta}_2 \leq R } + R \geq R + \epsilon } \\
&\quad \leq \Prb{ \norm{ \theta_{k+1} -  \theta_k }_2  \Ind{ \norm{ \theta_k - \bar \theta}_2 \leq R } \geq \epsilon }  \\
&\quad \leq \Prb{ \norm{ M_k \dot f(\theta_k, X_{k+1})}_2 \Ind{ \norm{ \theta_k - \bar \theta}_2 \leq R } \geq \epsilon } \\
&\quad \leq \frac{1}{\epsilon^2} \norm{M_k}_2^2 \E{ \condE{\norm{ \dot f(\theta_k, X_{k+1} )}_2^2 }{ \mathcal{F}_k} \Ind{ \norm{ \theta_k - \bar \theta}_2 \leq R }} \\
&\quad \leq \frac{1}{\epsilon^2} \norm{M_k}_2^2 \E{ G(\theta_k) \Ind{ \norm{ \theta_k - \bar \theta}_2 \leq R } } \\
&\quad \leq \frac{1}{\epsilon^2} \norm{M_k}_2^2 G_R,
\end{align}
where $G_R = \sup_{\theta: \norm{ \theta - \bar{\theta}}_2 \leq R } G(\theta) < \infty$ since $G$ is upper semi-continuous. By \cref{P2}, we see that the sum of the probabilities is finite. Together with the Borel-Cantelli lemma, $\Prb{ \norm{ \theta_{k+1} - \bar \theta}_2 \geq R + \epsilon, ~ \norm{ \theta_k - \bar \theta}_2 \leq R~i.o.} = 0$. Since $\epsilon > 0$ is arbitrary, we can show that this statement holds for a countable sequence of $\epsilon_n \downarrow 0$. As the union of countably many measure zero sets has measure zero, the conclusion of the result holds.
\end{proof}

\subsection{ Global Consequences of the Capture Theorem}

We begin with a direct consequence of \cref{theorem-capture}, which addresses the first component of \cref{theorem-convergence-local-holder-general-noise}.

\begin{corollary} \label{corollary-limits}
Suppose the setting of \cref{theorem-capture} holds. Let $\mathcal{A}_1 = \lbrace \liminf_{k \to \infty} \norm{ \theta_k}_2 = \infty \rbrace$ and $\mathcal{A}_2 = \lbrace \lim_{k \to \infty} \norm{ \theta_k}_2 < \infty \rbrace$. Then, $\Prb{ \mathcal{A}_1} + \Prb{ \mathcal{A}_2} = 1$.
\end{corollary}
\begin{proof}[Proof of \cref{corollary-limits}]
For any $R \geq 0 $, define the event
\begin{equation}
\mathcal{A}(R) = \left\lbrace \liminf_{k \to \infty} \norm{ \theta_k }_2 \leq R \right\rbrace \cap \left\lbrace \limsup_{k \to \infty} \norm{ \theta_k }_2 > R \right\rbrace.
\end{equation}
Then $\left\lbrace \liminf_{k \to \infty} \norm{ \theta_k }_2 < \limsup_{k \to \infty} \norm{ \theta_k }_2 \right\rbrace \subset \bigcup_{ R \in \mathbb{Q}_{\geq 0}} \mathcal{A}(R)$, where $\mathbb{Q}_{ \geq 0}$ is the set of non-negative rational numbers. By \cref{theorem-capture}, $\Prb{ \mathcal{A}(R) } = 0$ for all $R  \geq 0$. Since the countable union of measure zero sets has measure zero, $\Prb{ \liminf_{k \to \infty} \norm{ \theta_k }_2 < \limsup_{k \to \infty} \norm{ \theta_k }_2 } = 0$. Hence, we conclude that, with probability one, $\liminf_{k \to \infty} \norm{ \theta_k }_2$ is either finite and equal to $\limsup_{k \to \infty} \norm{ \theta_k }_2$ (i.e., $\mathcal{A}_2$) or is infinite (i.e., $\mathcal{A}_1$). Since $\mathcal{A}_1$ and $\mathcal{A}_2$ are mutually exclusive, the result follows.
\end{proof}

Importantly, \cref{corollary-limits} says that the possibility that the limit supremum and limit infimum of $\lbrace \norm{\theta_k}_2 \rbrace$ being distinct occurs with probability zero (cf., a simple random walk with a positive one bias at each step, which will have its limit supremum as infinity and limit infimum as zero with probability one). Thus, \cref{corollary-limits} provides us with two cases that we can study: (with probability one) $\lbrace \norm{ \theta _k}_2 \rbrace$ diverges or converges to a finite value. The remaining result explore what happens in the finite case.

\subsection{ Asymptotic Behavior of the Objective Function}

The following result applies \cref{lemma-ftc,lemma-variance-control} under \cref{A1,A2,A3a,A4a} to produce a recursive relationship between the objective function evaluated at two sequential iterates. 

\begin{lemma}\label{lemma-recursion-constrained}
Let $\lbrace M_k \rbrace$ be defined as in \eqref{eqn:iterates} satisfying \cref{P1}.
For all $k + 1 \in \mathbb{N}$ and $R \geq 0$, let $\mathcal{B}_k(R) = \bigcap_{j=0}^k \lbrace \norm{\theta_j}_2 \leq R \rbrace$. 
Suppose \cref{A1,A2,A3a,A4a} hold.
Then, $\forall R \geq 0$, $\exists L_{R+1} > 0$, such that
\begin{equation}
\begin{aligned}
&\condE{ [F(\theta_{k+1}) - \Flb] \Ind{ \mathcal{B}_{k+1}(R) } }{\mathcal{F}_k} \leq [F(\theta_k) - \Flb] \Ind{ \mathcal{B}_{k}(R) } \\
&\quad - \lambda_{\min}(M_k) \norm{ \dot{F}(\theta_k)}_2^2 \Ind{\mathcal{B}_k(R)} + \frac{L_{R+1} + \partial F_R}{1+\alpha} \lambda_{\max}(M_k)^{1+\alpha} G_R^{\frac{1+\alpha}{2}},
\end{aligned}
\end{equation}
where $G_R = \sup_{\theta \in \overline{B(R)}} G(\theta) < \infty$ with $G(\theta)$ defined in \cref{A3a}; and $\partial F_R  = \sup_{ \theta \in \overline{B(R)} } \inlinenorm{ \dot{F}(\theta)}_2 (1+\alpha) < \infty$.
\end{lemma}
\begin{proof}
Fix $R \geq 0$. For any $k +1 \in \mathbb{N}$, \cref{lemma-ftc} implies that $\exists L_{R+1} > 0$ such that
\begin{equation}
\begin{aligned}
&[F(\theta_{k+1}) - \Flb] \Ind{ \mathcal{B}_{k+1}(R+1) } \\
&\quad \leq \left( [F(\theta_k) - \Flb] + \dot{F}(\theta_k)'(\theta_{k+1} - \theta_k) + \frac{L_{R+1}}{1+\alpha} \norm{ \theta_{k+1} - \theta_k }_2^{1+\alpha} \right) \Ind{ \mathcal{B}_{k+1}(R+1) }.
\end{aligned}
\end{equation}

Now, since $\overline{B(R)} \subset \overline{B(R+1)}$, it also holds true that
\begin{equation}
\begin{aligned}
&[F(\theta_{k+1}) - \Flb] \Ind{ \mathcal{B}_{k+1}(R) } \\
&\quad \leq \left( [F(\theta_k) - \Flb] + \dot{F}(\theta_k)'(\theta_{k+1} - \theta_k) + \frac{L_{R+1}}{1+\alpha} \norm{ \theta_{k+1} - \theta_k }_2^{1+\alpha} \right) \Ind{ \mathcal{B}_{k+1}(R) }.
\end{aligned}
\end{equation}

Our goal now is to replace $\mathcal{B}_{k+1}(R)$ on the right hand side by $\mathcal{B}_{k}(R)$. However, there is a technical difficulty which we must address. First, it follows from the preceding inequality that
\begin{equation} \label{ineq-30sx4}
\begin{aligned}
&[F(\theta_{k+1}) - \Flb] \Ind{ \mathcal{B}_{k+1}(R) } \\
& \leq \left( [F(\theta_k) - \Flb] + \dot{F}(\theta_k)'(\theta_{k+1} - \theta_k) + \frac{L_{R+1}}{1+\alpha} \norm{ \theta_{k+1} - \theta_k }_2^{1+\alpha} \right) \\
&\quad \times \bigg{(} \Ind{ \mathcal{B}_{k+1}(R) } - \Ind{\mathcal{B}_k(R)} \bigg{)} \\
&\quad + \left( [F(\theta_k) - \Flb] + \dot{F}(\theta_k)'(\theta_{k+1} - \theta_k) + \frac{L_{R+1}}{1+\alpha} \norm{ \theta_{k+1} - \theta_k }_2^{1+\alpha} \right) \Ind{ \mathcal{B}_{k}(R) }.
\end{aligned}
\end{equation}

The first term on the right hand side of the inequality only contributes meaningfully if it is positive. Since $\Ind{ \mathcal{B}_{k}(R) } \geq \Ind{ \mathcal{B}_{k+1}(R) }$, then two statements hold: (i) $\Ind{\mathcal{B}_k(R)}\Ind{\mathcal{B}_{k+1}(R)} = \Ind{\mathcal{B}_{k+1}(R)}$; and (ii) the first term of the right hand side of \eqref{ineq-30sx4} is positive if and only if 
\begin{equation} \label{event-d49fa}
\left([F(\theta_k) - \Flb] + \dot{F}(\theta_k)'(\theta_{k+1} - \theta_k) + \frac{L_{R+1}}{1+\alpha} \norm{ \theta_{k+1} - \theta_k }_2^{1+\alpha}\right) \Ind{\mathcal{B}_k(R)} < 0.
\end{equation}
By the choice of $L_{R+1}$, \cref{A1} and \cref{lemma-ftc} imply that if \eqref{event-d49fa} occurs, then $\inlinenorm{\theta_{k+1}}_2 > R + 1 \geq \inlinenorm{\theta_k}_2 + 1$. By the reverse triangle inequality and \eqref{eqn:iterates}, if \eqref{event-d49fa} occurs, then $\inlinenorm{ M_k \dot{f}(\theta_k, X_{k+1})}_2 \geq 1$. Hence,
\begin{equation}
\begin{aligned}
&\left( [F(\theta_k) - \Flb] + \dot{F}(\theta_k)'(\theta_{k+1} - \theta_k) + \frac{L_{R+1}}{1+\alpha} \norm{ \theta_{k+1} - \theta_k }_2^{1+\alpha} \right) \\
&\quad \times \bigg{(} \Ind{ \mathcal{B}_{k+1}(R) } - \Ind{\mathcal{B}_k(R)} \bigg{)} \\
&\leq \left( -[F(\theta_k) - \Flb] - \dot{F}(\theta_k)'(\theta_{k+1} - \theta_k) - \frac{L_{R+1}}{1+\alpha} \norm{ \theta_{k+1} - \theta_k }_2^{1+\alpha} \right) \\
&\quad \times \bigg{(} \Ind{ \mathcal{B}_{k}(R) } - \Ind{\mathcal{B}_{k+1}(R)} \bigg{)} \Ind{ \mathcal{B}_k(R)} \Ind{ \norm{M_k \dot f (\theta_k, X_{k+1} )}_2 \geq 1 }.
\end{aligned}
\end{equation}

We now compute another coarse upper bound for this inequality. Note, by \cref{A1} and Cauchy-Schwarz, 
\begin{align}
&\begin{aligned}
&\left( -[F(\theta_k) - \Flb] - \dot{F}(\theta_k)'(\theta_{k+1} - \theta_k) - \frac{L_{R+1}}{1+\alpha} \norm{ \theta_{k+1} - \theta_k }_2^{1+\alpha} \right) \\
&\quad \times \bigg{(} \Ind{ \mathcal{B}_{k}(R) } - \Ind{\mathcal{B}_{k+1}(R)} \bigg{)} \Ind{ \mathcal{B}_k(R)} \Ind{ \norm{M_k \dot f (\theta_k, X_{k+1} )} _2 \geq 1 }
\end{aligned} \\
&\quad \leq \norm{ \dot F(\theta_k) }_2 \norm{ M_k \dot f (\theta_k, X_{k+1} ) }_2 \Ind{ \mathcal{B}_k(R)} \Ind{ \norm{M_k \dot f (\theta_k, X_{k+1} )}_2 \geq 1 } \\
&\quad \leq \norm{ \dot F(\theta_k) }_2 \norm{ M_k \dot f (\theta_k, X_{k+1} ) }_2^{1+\alpha}\Ind{ \mathcal{B}_k(R)} \\
&\quad \leq \frac{\partial F_R}{1+\alpha}  \norm{ M_k \dot f (\theta_k, X_{k+1} ) }_2^{1+\alpha}\Ind{ \mathcal{B}_k(R)},
\end{align}
where $\partial F_R = \sup_{ \theta \in \overline{B(R)} } \inlinenorm{ \dot{F}(\theta)}_2 (1 + \alpha) < \infty$ given that $\inlinenorm{\dot{F}(\theta)}_2$ is a continuous function of $\theta$.

Applying this inequality to \eqref{ineq-30sx4}, we conclude
\begin{equation}
\begin{aligned}
&[F(\theta_{k+1}) - \Flb] \Ind{ \mathcal{B}_{k+1}(R) } \\
& \leq \left( [F(\theta_k) - \Flb] - \dot{F}(\theta_k)'M_k \dot f(\theta_k, X_{k+1}) + \frac{L_{R+1} + \partial F_R}{1+\alpha} \norm{ M_k \dot f(\theta_k, X_{k+1}) }_2^{1+\alpha} \right) \\
& \quad \times \Ind{ \mathcal{B}_{k}(R) }.
\end{aligned}
\end{equation}

By \cref{A2},
\begin{equation}
\begin{aligned}
&\condE{[F(\theta_{k+1}) - \Flb] \Ind{ \mathcal{B}_{k+1}(R) }}{\mathcal{F}_k} \\
&\leq \left( [F(\theta_k) - \Flb] - \dot{F}(\theta_k)'M_k \dot{F}(\theta_k) + \frac{L_{R+1} + \partial F_R}{1+\alpha}\condE{ \norm{ M_k \dot{f}(\theta_k,X_{k+1}) }_2^{1+\alpha} }{\mathcal{F}_k} \right)\\
&\quad \times \Ind{ \mathcal{B}_{k}(R) }.
\end{aligned}
\end{equation}

Using \cref{P1}, \cref{A3a} and \cref{lemma-variance-control},
\begin{equation}
\begin{aligned}
&\condE{[F(\theta_{k+1}) - \Flb] \Ind{ \mathcal{B}_{k+1}(R) }}{\mathcal{F}_k} \\
&\leq \left( [F(\theta_k) - \Flb] - \lambda_{\min}(M_k)\norm{\dot{F}(\theta_k)}_2^2 + \frac{L_{R+1} + \partial F_R}{1+\alpha} \lambda_{\max}(M_k)^{1+\alpha} G(\theta_k)^\frac{1+\alpha}{2}  \right) \Ind{ \mathcal{B}_{k}(R) }.
\end{aligned}
\end{equation}

By \cref{A3a}, $G$ is upper semicontinuous and $\overline{B(R)}$ is compact, which implies that $G_R$ is well defined and finite.
The result follows.
\end{proof}

The following corollary to \cref{lemma-recursion-constrained} proves the convergence of the objective function component for \cref{theorem-convergence-local-holder-general-noise}.
\begin{corollary} \label{corollary-convergence-objective-constrained}
Let $\lbrace \theta_k \rbrace$ be defined as in \eqref{eqn:iterates} satisfying \cref{P1,P2}. 
Suppose \cref{A1,A2,A3a,A4a} hold.
Then, there exists a finite random variable $F_{\lim}$ such that on the event $\lbrace \sup_{k} \norm{\theta_k}_2 < \infty \rbrace$, $\lim_{k \to \infty} F(\theta_k) = F_{\lim}$ with probability one.
\end{corollary}
\begin{proof}
By \cref{lemma-recursion-constrained}, for every $R \geq 0$,
\begin{equation}
\begin{aligned}
&\condE{ [F(\theta_{k+1}) - \Flb] \Ind{ \mathcal{B}_{k+1}(R) } }{\mathcal{F}_k} \\
&\quad \leq [F(\theta_k) - \Flb] \Ind{ \mathcal{B}_{k}(R) }
+ \frac{(L_{R+1} + \partial F_R) G_R^{\frac{1+\alpha}{2}}}{1+\alpha} \lambda_{\max}(M_k)^{1+\alpha}.
\end{aligned}
\end{equation}
By \citet[Exercise II.4]{neveu1975} (cf. \citet*{robbins1971}) and \cref{P2}, $\lim_{k \to \infty} [ F(\theta_k) - \Flb] \Ind{ \mathcal{B}_{k}(R)}$ converges to a finite random variable with probability one. Since $\Flb$ is a constant and $R \geq 0$ is arbitrary, we conclude that there exists a finite random variable $F_{\lim}$ such that $ \lbrace \sup_{k} \norm{\theta_k}_2 \leq R \rbrace \subset \lbrace \lim_{k} F(\theta_k) = F_{\lim} \rbrace$ up to a measure zero set. Since the countable union of measure zero sets has measure zero,
\begin{equation}
\left\lbrace \sup_k \norm{\theta_k}_2 < \infty \right\rbrace = \bigcup_{ R \in \mathbb{N} } \left\lbrace \sup_k \norm{ \theta_k}_2 \leq R \right\rbrace \subset \left\lbrace \lim_{k \to \infty} F(\theta_k) = F_{\lim} \right\rbrace,
\end{equation}
up to a measure zero set. The result follows.
\end{proof}

\subsection{Asymptotic Behavior of the Gradient}
We now prove that the gradient norm evaluated at SGD's iterates must, repeatedly, get arbitrarily close to zero. We adapt the strategy of \citet{patel2020}.
\begin{lemma} \label{lemma-liminf-grad-constrained}
Let $\lbrace \theta_k \rbrace$ be defined as in \eqref{eqn:iterates} satisfying \cref{P1,P2,P3}. 
For all $k + 1 \in \mathbb{N}$ and $R \geq 0$, let $\mathcal{B}_k(R) = \bigcap_{j=0}^\infty \lbrace \norm{\theta_k}_2 \leq R \rbrace$.
Suppose \cref{A1,A2,A3a,A4a} hold. Then, $\forall R \geq 0$ and for all $\delta > 0$,
\begin{equation}
\condPrb{ \norm{ \dot{F}(\theta_k)}_2^2 \Ind{\mathcal{B}_k(R)} \leq \delta, ~i.o. }{\mathcal{F}_0} = 1,~ w.p.1.
\end{equation}
\end{lemma}
\begin{proof}
By \cref{lemma-recursion-constrained},
\begin{equation}
\begin{aligned}
&\lambda_{\min}(M_k) \E{\norm{ \dot{F}(\theta_k)}_2^2 \Ind{\mathcal{B}_k(R)}} \leq \E{[F(\theta_k) - \Flb] \Ind{ \mathcal{B}_{k}(R) }} \\
&\quad - \E{ [F(\theta_{k+1}) - \Flb] \Ind{ \mathcal{B}_{k+1}(R) } } + \frac{(L_{R+1} + \partial F_R) G_R^{\frac{1+\alpha}{2}}}{1+\alpha} \lambda_{\max}(M_k)^{1+\alpha}.
\end{aligned}
\end{equation}
Taking the sum of this equation for all $k$ from $0$ to $j \in \mathbb{N}$, we have
\begin{equation}
\begin{aligned}
& \sum_{k=0}^j \lambda_{\min}(M_k) \condE{\norm{ \dot{F}(\theta_k)}_2^2 \Ind{\mathcal{B}_k(R)}}{\mathcal{F}_0} \leq [F(\theta_0) - \Flb] \Ind{ \mathcal{B}_{0}(R) } \\
&\quad - \condE{ [F(\theta_{j+1}) - \Flb] \Ind{ \mathcal{B}_{j+1}(R) } }{\mathcal{F}_0} + \frac{(L_{R+1} + \partial F_R) G_R^{\frac{1+\alpha}{2}}}{1+\alpha} \sum_{k=0}^j \lambda_{\max}(M_k)^{1+\alpha}.
\end{aligned}
\end{equation}
By \cref{A1,P2}, the right hand side is bounded by
\begin{equation}
[F(\theta_0) - \Flb] \Ind{ \mathcal{B}_{0}(R) } +  \frac{(L_{R+1} + \partial F_R) G_R^{\frac{1+\alpha}{2}}}{1+\alpha}S,
\end{equation}
which is finite with probability one. Therefore, $\sum_{k=0}^\infty \lambda_{\min}(M_k) \inlinecondE{ \inlinenorm{ \dot{F}(\theta_k)}_2^2 \Ind{\mathcal{B}_k(R)}}{\mathcal{F}_0}$ is finite almost surely. Furthermore, by \cref{P3}, $\liminf_k \inlinecondE{ \inlinenorm{ \dot{F}(\theta_k)}_2^2 \Ind{\mathcal{B}_k(R)}}{\mathcal{F}_0} = 0$ with probability one.

Now, for any $\delta > 0$, Markov's inequality implies that for all $j +1 \in \mathbb{N}$,
\begin{equation}
\condPrb{ \bigcap_{k=j}^\infty \left\lbrace \norm{ \dot{F}(\theta_k)}_2^2 \Ind{\mathcal{B}_k(R)} > \delta \right\rbrace }{\mathcal{F}_0} \leq \frac{1}{\delta} \min_{j \leq k} \condE{ \norm{ \dot{F}(\theta_k)}_2^2 \Ind{\mathcal{B}_k(R)} }{\mathcal{F}_0},
\end{equation}
where the right hand side is zero with probability one because $\liminf_k \inlinecondE{ \inlinenorm{ \dot{F}(\theta_k)}_2^2 \Ind{\mathcal{B}_k(R)}}{\mathcal{F}_0} = 0$ with probability one.

As the countable union of measure zero sets has measure zero, we conclude that for all $\delta > 0$,
\begin{equation}
\condPrb{ \norm{ \dot{F}(\theta_k)}_2^2 \Ind{\mathcal{B}_k(R)} \leq \delta, ~i.o. }{\mathcal{F}_0} = 1,
\end{equation}
with probability one.
\end{proof}

Unfortunately, \cref{lemma-liminf-grad-constrained} does not guarantee that the gradient norm will be captured within a region of zero. In order to prove this, we first show that it is not possible (i.e., a zero probability event) for the limit supremum and limit infimum of the gradients to be distinct (cf., \cref{theorem-capture} for iterate distances).
\begin{lemma} \label{lemma-limsup-grad-constrained}
Let $\lbrace \theta_k \rbrace$ be defined as in \eqref{eqn:iterates} satisfying \cref{P1,P2}. 
For all $k + 1 \in \mathbb{N}$ and $R \geq 0$, let $\mathcal{B}_k(R) = \bigcap_{j=0}^\infty \lbrace \norm{\theta_k}_2 \leq R \rbrace$.
Suppose \cref{A1,A2,A3a,A4a} hold. Then, $\forall R \geq 0$ and for all $\delta > 0$,
\begin{equation}
\condPrb{  \norm{ \dot{F}(\theta_{k+1})}_2 \Ind{ \mathcal{B}_{k+1}(R) } > \delta, \norm{ \dot{F}(\theta_k)}_2 \Ind{ \mathcal{B}_k(R) } \leq \delta , ~i.o.}{\mathcal{F}_0} = 0,
\end{equation}
with probability one.
\end{lemma}
\begin{proof}
Let $\epsilon > 0$, $L_R > 0$, and $G_R$ be defined as in \cref{lemma-recursion-constrained}. Then, for $\delta > 0$,
\begin{align}
&\condPrb{ \norm{ \dot{F}(\theta_{k+1})}_2 \Ind{ \mathcal{B}_{k+1}(R) } \Ind{ \norm{ \dot{F}(\theta_k)}_2 \Ind{ \mathcal{B}_k(R) } \leq \delta } > \delta + L_R \epsilon^\alpha }{\mathcal{F}_0} \\
&= \mathbb{P} \bigg{[} \left( \norm{ \dot{F}(\theta_{k+1})}_2 - \norm{\dot{F}(\theta_k)}_2 + \norm{\dot F(\theta_k)}_2 \right) \Ind{ \mathcal{B}_{k+1}(R) }  \\
&\quad \times \Ind{ \norm{ \dot{F}(\theta_k)}_2 \Ind{ \mathcal{B}_k(R) } \leq \delta } > \delta + L_R \epsilon^\alpha \bigg{\vert}  \mathcal{F}_0 \bigg{]} \\
&\leq \mathbb{P} \bigg{[} L_R \norm{ \theta_{k+1} - \theta_k}_2^\alpha \Ind{ \mathcal B_{k+1}(R) }\Ind{ \norm{ \dot{F}(\theta_k)}_2 \Ind{ \mathcal{B}_k(R) } \leq \delta } > L_R \epsilon^\alpha \bigg{\vert} \mathcal{F}_0 \bigg{]} \\
&= \mathbb{P} \bigg{[} \norm{ M_k \dot f(\theta_k, X_{k+1} )}_2 \Ind{ \mathcal B_{k+1}(R) }\Ind{ \norm{ \dot{F}(\theta_k)}_2 \Ind{ \mathcal{B}_k(R) } \leq \delta } > \epsilon\bigg{\vert} \mathcal{F}_0 \bigg{]} \\
&\leq \condPrb{ \norm{M_k \dot f(\theta_k, X_{k+1} ) }_2 \Ind{ \mathcal B_k(R) } > \epsilon}{\mathcal{F}_0} \\
&\leq \frac{1}{\epsilon^2} \norm{M_k}_2^2 \condE{ \norm{ \dot f(\theta_k, X_{k+1}) }_2^2 \Ind{ \mathcal B_k(R)} }{\mathcal{F}_0} \\
&\leq \frac{1}{\epsilon^2} \norm{M_k}_2^2 G_R.
\end{align}
By \cref{P2}, the sum of the last expression over all $k+1 \in \mathbb{N}$ is finite. By the Borel-Cantelli lemma, for all $R \geq 0$, $\delta > 0$ and $\epsilon > 0$,
\begin{equation}
\condPrb{  \norm{ \dot{F}(\theta_{k+1})}_2 \Ind{ \mathcal{B}_{k+1}(R) } > \delta + L_R \epsilon^\alpha, \norm{ \dot{F}(\theta_k)}_2 \Ind{ \mathcal{B}_k(R) } \leq \delta , ~i.o.}{\mathcal{F}_0} = 0,
\end{equation}
with probability one.
Since this holds for any $\epsilon > 0$, it will hold for every value in a sequence $\epsilon_n \downarrow 0$. Since the countable union of measure zero events has measure zero,
\begin{equation}
\condPrb{  \norm{ \dot{F}(\theta_{k+1})}_2 \Ind{ \mathcal{B}_{k+1}(R) } > \delta, \norm{ \dot{F}(\theta_k)}_2 \Ind{ \mathcal{B}_k(R) } \leq \delta , ~i.o.}{\mathcal{F}_0} = 0,
\end{equation}
with probability one.
\end{proof}

We now put together \cref{lemma-liminf-grad-constrained,lemma-limsup-grad-constrained} to show that, on the event $\lbrace \sup_k \inlinenorm{\theta_k}_2 < \infty \rbrace$, $\inlinenorm{ \dot{F}(\theta_k) }_2 \to 0$ with probability one.

\begin{corollary} \label{corollary-convergence-grad-constrained}
Let $\lbrace \theta_k \rbrace$ be defined as in \eqref{eqn:iterates} satisfying \cref{P1,P2,P3}. 
Suppose \cref{A1,A2,A3a,A4a} hold. Then, on the event $\lbrace \sup_k \inlinenorm{\theta_k}_2 < \infty \rbrace$, $\lim_{k \to \infty} \inlinenorm{ \dot{F}(\theta_k)}_2 = 0$ with probability one.
\end{corollary}
\begin{proof}
For any $R \geq 0$ and $\delta > 0$, \cref{lemma-liminf-grad-constrained} implies
\begin{equation}
\begin{aligned}
&\condPrb{ \norm{\dot{F}(\theta_{k+1})}_2 \Ind{ \mathcal{B}_{k+1}(R) } > \delta, ~i.o. }{\mathcal{F}_0} \\
&\quad = \condPrb{ \left\lbrace \norm{\dot{F}(\theta_{k+1})}_2 \Ind{ \mathcal{B}_{k+1}(R) } > \delta \right\rbrace \cap \left\lbrace \norm{\dot{F}(\theta_{k})}_2 \Ind{ \mathcal{B}_{k}(R) } \leq \delta, ~i.o. \right\rbrace }{\mathcal{F}_0},
\end{aligned}
\end{equation}
with probability one. We see that this latter event is exactly, 
\begin{equation}
\condPrb{  \norm{ \dot{F}(\theta_{k+1})}_2 \Ind{ \mathcal{B}_{k+1}(R) } > \delta, \norm{ \dot{F}(\theta_k)}_2 \Ind{ \mathcal{B}_k(R) } \leq \delta , ~i.o.}{\mathcal{F}_0},
\end{equation}
which, by \cref{lemma-limsup-grad-constrained}, is zero with probability one. Therefore, $\condPrb{ \inlinenorm{\dot F(\theta_{k+1})}_2 \inlineInd{ \mathcal{B}_{k+1}(R) } > \delta, ~i.o. }{\mathcal{F}_0}$ is zero with probability one. Letting $\delta_n \downarrow 0$ and noting that the countable union of measure zero sets has measure zero, we conclude $\condPrb{ \inlinenorm{\dot F(\theta_{k+1})}_2 \inlineInd{ \mathcal{B}_{k+1}(R) } > 0, ~i.o. }{\mathcal{F}_0} = 0$ with probability one.

Therefore, for all $R \geq 0$, $\lbrace \sup_k \norm{\theta_k}_2 \leq R \rbrace \subset \lbrace \lim_{k\to\infty} \inlinenorm{\dot{F}(\theta_k)}_2 = 0 \rbrace$ up to a measure zero set. Since $\lbrace \sup_k \inlinenorm{\theta_k}_2 < \infty \rbrace = \cup_{R \in \mathbb{N} } \lbrace \sup_k \inlinenorm{\theta_k}_2 \leq R \rbrace$, the result follows.
\end{proof}

\section{Analysis of the Global H\"{o}lder Continuity and Expected Smoothness Case} \label{section-analysis-infinite}
We will divide the proof into four pieces. In \cref{subsection-global-holder-objective}, we will begin by proving that $\lbrace F(\theta_k) \rbrace$ converges to an \textit{integrable} random variable with probability one, which follows the same strategy used for \cref{theorem-convergence-local-holder-general-noise}. In \cref{subsection-global-holder-expected-objective}, we will then prove that $\lbrace \inlineE{ F(\theta_k) } \rbrace$ are bounded, which is an alternative way to imply that $F_{\lim}$ is integrable via Fatou's lemma and which implies the $L^1$ convergence of $F(\theta_k)^\gamma$ to $F_{\lim}^\gamma$ for $\gamma \in [0,1)$ by H\"{o}lder's inequality and uniform integrability. In \cref{subsection-global-holder-gradient}, we will prove that $\lbrace \inlinenorm{ \dot{F}(\theta_k)}_2 \rbrace$ converges to zero with probability one. Finally, in \cref{subsection-global-holder-expected-gradient}, we will prove that $\sup_{k} \inlineE{ \inlinenorm{ \dot{F}(\theta_k)}_2^2} < \infty$, from which we can conclude that $\inlineE{ \inlinenorm{\dot F(\theta_k)}_2} \to 0$ as $k \to \infty$.

\subsection{Asymptotic Behavior of the Objective Function} \label{subsection-global-holder-objective}
We begin with an analogue of \cref{lemma-recursion-constrained} that allows us to use the global H\"{o}lder assumption to remove the indicator function that burdened \cref{lemma-recursion-constrained}

\begin{lemma} \label{lemma-recursion-unconstrained}
Let $\lbrace \theta_k \rbrace$ be defined as in \eqref{eqn:iterates} satisfying \cref{P1}.
Suppose \cref{A1,A2,A3b,A4b} hold. Then,
\begin{equation}
\begin{aligned}
&\condE{F(\theta_{k+1}) - \Flb}{\mathcal{F}_k} 
\leq [ F(\theta_k) - \Flb ] \left( 1 + \frac{LC_2}{2} \lambda_{\max}(M_k)^{1+\alpha} \right) \\
&\quad - \norm{ \dot{F}(\theta_k)}_2^2 \left( \lambda_{\min}(M_k) - \frac{L C_3}{2} \lambda_{\max}(M_k)^{1+\alpha} \right) \\
&\quad+ \frac{L}{1+\alpha}\lambda_{\max}(M_k)^{1+\alpha}\left(\frac{1+\alpha}{2}C_1+\frac{1-\alpha}{2} \right)
\end{aligned}
\end{equation}
If, in addition, \cref{P4} holds then $\exists K \in \mathbb{N}$ such that for all $k \geq K$, 
\begin{equation}
\begin{aligned}
&\condE{F(\theta_{k+1}) - \Flb}{\mathcal{F}_k} 
\leq [ F(\theta_k) - \Flb ] \left( 1 + \frac{LC_2}{2} \lambda_{\max}(M_k)^{1+\alpha} \right) \\
&\quad - \frac{1}{2}\lambda_{\min}(M_k) \norm{ \dot{F}(\theta_k)}_2^2 + \frac{L}{1+\alpha}\lambda_{\max}(M_k)^{1+\alpha}\left(\frac{1+\alpha}{2}C_1+\frac{1-\alpha}{2} \right)
\end{aligned}
\end{equation}

\end{lemma}
\begin{proof}
By \cref{lemma-ftc} and \eqref{eqn:iterates},
\begin{equation}
F(\theta_{k+1})- \Flb \leq F(\theta_{k})-\Flb -\dot{F} (\theta_{k})^{'}M_k \dot{f}(\theta_k,X_{k+1})+\frac{L}{1+\alpha} \norm{ M_k \dot{f}(\theta_k,X_{k+1})}_2^{1+\alpha}
\end{equation}

Now, taking conditional expectations, applying \cref{A2,A3b} and \cref{lemma-variance-control},
\begin{equation}
\begin{aligned}
&\condE{ F(\theta_{k+1}) - \Flb }{\mathcal{F}_k} \\
&\leq F(\theta_k) - \Flb - \dot F (\theta_k)' M_k \dot F(\theta_k) \\
& + \frac{L}{1+\alpha}\lambda_{\max}(M_k)^{1+\alpha}\left[\left(\frac{1+\alpha}{2}\right)\left( C_1 + C_2 (F(\theta_k) - \Flb) + C_3 \norm{ \dot F(\theta_k)}_2^2 \right) +\frac{1-\alpha}{2} \right]
\end{aligned}
\end{equation}

Finally, using \cref{P1} to show $ - \dot F (\theta_k) ' M_k \dot F(\theta_k) \leq - \lambda_{\min}(M_k) \inlinenorm{ \dot F (\theta_k) }_2^2$ and rearranging the terms, the first part of the result follows.

By \cref{Lemma4}, there exists $K \in \mathbb{N}$ such that for all $k \geq K$,
\begin{equation}
\lambda_{\min}(M_k) - \frac{L C_3}{2} \lambda_{\max}(M_k)^{1+\alpha} \geq \frac{1}{2} \lambda_{\min}(M_k).
\end{equation}
The result follows.
\end{proof}

\begin{corollary} \label{corollary-convergence-objective-unconstrained}
Let $\lbrace \theta_k \rbrace$ be defined as in \eqref{eqn:iterates} satisfying \cref{P1,P2,P4}. 
Suppose \cref{A1,A2,A3b,A4b} hold.
Then, there exists an integrable random variable $F_{\lim}$ such that $\lim_{k \to \infty} F(\theta_k) = F_{\lim}$ with probability one.
\end{corollary}
\begin{proof}
\cref{lemma-recursion-unconstrained} implies $\exists K \in \mathbb{N}$ such that $k \geq K$, 
\begin{equation}
\begin{aligned}
&\condE{F(\theta_{k+1}) - \Flb}{\mathcal{F}_k} 
\leq [ F(\theta_k) - \Flb ] \left( 1 + \frac{LC_2}{2} \lambda_{\max}(M_k)^{1+\alpha} \right) \\
&\quad+ \frac{L}{1+\alpha}\lambda_{\max}(M_k)^{1+\alpha}\left(\frac{1+\alpha}{2}C_1+\frac{1-\alpha}{2} \right)
\end{aligned}
\end{equation}

By \citet[Exercise II.4]{neveu1975} (cf. \citet*{robbins1971}) and \cref{P2}, $\lim_{k \to \infty} [ F(\theta_k) - \Flb]$ converges to an integrable random variable with probability one. The result follows.
\end{proof}

\subsection{Asymptotic Behavior of the Expected Objective Function} \label{subsection-global-holder-expected-objective}

We now follow \citet{lei2019} to prove that expected value of the objective function evaluated at the iterates remains bounded. This relies on the following recursive relationship.

\begin{lemma} \label{lemma-recursion-expected-unconstrained}
Let $\lbrace \theta_k \rbrace$ be defined as in \eqref{eqn:iterates} satisfying \cref{P1,P2,P4}. 
Suppose \cref{A1,A2,A3b,A4b} hold. 
There exists a $K \in \mathbb{N}$ such that for all $k \geq K$,
\begin{equation}
\begin{aligned}
&\condE{F(\theta_{k+1}) - \Flb}{\mathcal{F}_k} + \left[ \frac{LC_1}{2}+\frac{L}{2}\left(\frac{1-\alpha}{1+\alpha}\right)\right] \sum_{j=k+1}^{\infty} \lambda_{\max}(M_j)^{1+\alpha}  \\
&\leq \exp\left(\frac{LC_2}{2}\lambda_{\max}(M_k)^{1+\alpha}\right)\left[ F(\theta_k) - \Flb + \left[ \frac{LC_1}{2}+\frac{L}{2}\left(\frac{1-\alpha}{1+\alpha}\right)\right] \sum_{j=k}^{\infty} \lambda_{\max}(M_j)^{1+\alpha} \right]
\end{aligned}
\end{equation}
\end{lemma}
\begin{proof}
\cref{lemma-recursion-unconstrained} implies $\exists K \in \mathbb{N}$ such that $k \geq K$, 
\begin{equation}
\begin{aligned}
&\condE{F(\theta_{k+1}) - \Flb}{\mathcal{F}_k} 
\leq [ F(\theta_k) - \Flb ] \left( 1 + \frac{LC_2}{2} \lambda_{\max}(M_k)^{1+\alpha} \right) \\
&\quad+ \frac{L}{1+\alpha}\lambda_{\max}(M_k)^{1+\alpha}\left(\frac{1+\alpha}{2}C_1+\frac{1-\alpha}{2} \right) \left( 1 + \frac{LC_2}{2} \lambda_{\max}(M_k)^{1+\alpha} \right)
\end{aligned}
\end{equation}

Since $1 + x \leq \exp(x)$ for $x \geq 0$, 
\begin{equation}
\begin{aligned}
&\condE{F(\theta_{k+1}) - \Flb}{\mathcal{F}_k} 
\leq \exp\left(\frac{LC_2}{2}\lambda_{\max}(M_k)^{1+\alpha}\right) \\
& \quad \times  \left(  F(\theta_k) - \Flb  + \left[ \frac{LC_1}{2}+\frac{L}{2}\left(\frac{1-\alpha}{1+\alpha}\right)\right] \lambda_{\max}(M_k)^{1+\alpha} \right)
\end{aligned}
\end{equation}

The result follows by \cref{P2}, adding
\begin{equation}
\left[ \frac{LC_1}{2}+\frac{L}{2}\left(\frac{1-\alpha}{1+\alpha}\right)\right] \sum_{j=k+1}^\infty \lambda_{\max}(M_j)^{1+\alpha}
\end{equation}
to both sides, and noting that $\exp(x) \geq 1$ for $x \geq 0$.
\end{proof}

\begin{corollary} \label{corollary-convergence-expected-objective-unconstrained}
Let $\lbrace \theta_k \rbrace$ be defined as in \eqref{eqn:iterates} satisfying \cref{P1,P2,P4}. 
Suppose \cref{A1,A2,A3b,A4b} hold.
Then, $\sup_{k} \inlinecondE{ F(\theta_k) }{\mathcal{F}_0} < \infty$ with probability one.
Finally, for any $\gamma \in [0,1)$, $\lim_{k \to \infty} \inlinecondE{ | (F(\theta_k) - \Flb)^\gamma - (F_{\lim}-\Flb)^\gamma | }{\mathcal{F}_0} = 0$ with probability one.
\end{corollary}
\begin{proof}
Applying \cref{lemma-recursion-expected-unconstrained} recursively,
\begin{equation}
\begin{aligned}
&\condE{F(\theta_{k+1}) - \Flb}{\mathcal{F}_K} + \left[ \frac{LC_1}{2}+\frac{L}{2}\left(\frac{1-\alpha}{1+\alpha}\right)\right] \sum_{j=k+1}^{\infty} \lambda_{\max}(M_j)^{1+\alpha}  \\
&\leq \exp\left(\frac{LC_2}{2} \sum_{j=K}^k \lambda_{\max}(M_j)^{1+\alpha}\right) \\
&\quad \times \left[ F(\theta_K) - \Flb + \left[ \frac{LC_1}{2}+\frac{L}{2}\left(\frac{1-\alpha}{1+\alpha}\right)\right] \sum_{j=K}^{\infty} \lambda_{\max}(M_j)^{1+\alpha} \right].
\end{aligned}
\end{equation}

By \cref{P2} and given that $K \in \mathbb{N}$ is a constant,
\begin{equation}
\begin{aligned}
&\condE{F(\theta_{k+1}) - \Flb}{\mathcal{F}_0} + \left[ \frac{LC_1}{2}+\frac{L}{2}\left(\frac{1-\alpha}{1+\alpha}\right)\right] \sum_{j=k+1}^{\infty} \lambda_{\max}(M_j)^{1+\alpha}  \\
&\leq \exp\left(\frac{LC_2}{2}S \right) \left[ \condE{F(\theta_K) - \Flb}{\mathcal{F}_0} + \left[ \frac{LC_1}{2}+\frac{L}{2}\left(\frac{1-\alpha}{1+\alpha}\right)\right] S \right],
\end{aligned}
\end{equation}
for which the right hand side is finite with probability one. Hence, $\sup_{k} \inlinecondE{ F(\theta_k) }{\mathcal{F}_0} < \infty$ with probability one. (Note, we can now apply Fatou's lemma to prove $\E{ F_{\lim}} < \infty$, if it were not already provided for in \citet{neveu1975}.)

For the final part of the proof, we note that $\gamma = 0$ is trivial. So, take $\gamma \in (0,1)$. Then, $\lbrace (F(\theta_k) - \Flb)^{\gamma} \rbrace$ are bounded in $L^{1/\gamma}$ (condition on $\mathcal{F}_0$), as we have just shown. Thus, $\lbrace (F(\theta_k) - \Flb)^\gamma \rbrace$ are uniformly integrable and, by \cref{corollary-convergence-objective-unconstrained}, $\lbrace (F(\theta_k) - \Flb)^\gamma \rbrace$ converges to $(F_{\lim}-\Flb)^\gamma$ in $L^1$. 
\end{proof}

\subsection{Asymptotic Behavior of the Gradient Function} \label{subsection-global-holder-gradient}
Just as we did before, we now prove that the gradient norm evaluated at SGD's iterates must, repeatedly, get arbitrarily close to zero. We use the strategy of \citet{patel2020}.

\begin{lemma} \label{lemma-liminf-grad-unconstrained}
Let $\lbrace \theta_k \rbrace$ be defined as in \eqref{eqn:iterates} satisfying \cref{P1,P2,P3,P4}.
Suppose \cref{A1,A2,A3b,A4b} hold.
Then, for all $\delta > 0$.
\begin{equation}
\condPrb{ \norm{ \dot{F}(\theta_k)}_2^2  \leq \delta, ~i.o. }{\mathcal{F}_0} = 1,~ w.p.1.
\end{equation}
\end{lemma}
\begin{proof}
By \cref{lemma-recursion-unconstrained}, there exists $K \in \mathbb{N}$ such that for all $k \geq K$,
\begin{equation}
\begin{aligned}
&\frac{1}{2}\lambda_{\min}(M_k) \norm{ \dot{F}(\theta_k)}_2^2 
\leq [ F(\theta_k) - \Flb ] - \condE{F(\theta_{k+1}) - \Flb}{\mathcal{F}_k}  \\
&\quad + [ F(\theta_k) - \Flb ]\frac{LC_2}{2} \lambda_{\max}(M_k)^{1+\alpha} 
+ \frac{L}{1+\alpha}\lambda_{\max}(M_k)^{1+\alpha}\left(\frac{1+\alpha}{2}C_1+\frac{1-\alpha}{2} \right).
\end{aligned}
\end{equation}

Now, taking expectation with respect to $\mathcal{F}_0$ and applying \cref{corollary-convergence-expected-objective-unconstrained} with $M_0 := \sup_{k} \inlinecondE{ F(\theta_k) - \Flb }{\mathcal{F}_0}$, 
\begin{equation}
\begin{aligned}
&\frac{1}{2}\lambda_{\min}(M_k) \condE{\norm{ \dot{F}(\theta_k)}_2^2}{\mathcal{F}_0} 
\leq \condE{ F(\theta_k) - \Flb }{\mathcal{F}_0} - \condE{F(\theta_{k+1}) - \Flb}{\mathcal{F}_0}  \\
&\quad + M_0\frac{LC_2}{2} \lambda_{\max}(M_k)^{1+\alpha} 
+ \frac{L}{1+\alpha}\lambda_{\max}(M_k)^{1+\alpha}\left(\frac{1+\alpha}{2}C_1+\frac{1-\alpha}{2} \right).
\end{aligned}
\end{equation}

Summing from over all $k \geq K$, using \cref{P2}, and \cref{A1},
\begin{equation}
\begin{aligned}
&\frac{1}{2} \sum_{k=K}^\infty \lambda_{\min}(M_k) \condE{\norm{ \dot{F}(\theta_k)}_2^2}{\mathcal{F}_0} 
\leq M_0 +  \frac{SL}{2} \left[ M_0 C_2 + C_1+\frac{1-\alpha}{1+\alpha} \right].
\end{aligned}
\end{equation}
Therefore, by \cref{P3}, we conclude that $\liminf_k \inlinecondE{ \inlinenorm{ \dot F(\theta_k) }_2^2}{\mathcal{F}_0} = 0$ with probability one.

Now, for any $\delta > 0$, Markov's inequality implies that for all $j \geq K$,
\begin{equation}
\condPrb{ \bigcap_{k=j}^\infty \left\lbrace \norm{ \dot{F}(\theta_k)}_2^2  > \delta \right\rbrace }{\mathcal{F}_0} \leq \frac{1}{\delta} \min_{j \leq k} \condE{ \norm{ \dot{F}(\theta_k)}_2^2  }{\mathcal{F}_0} = 0,
\end{equation}
with probability one. The countable union of measure zero sets has measure zero. Therefore, the conclusion follows.
\end{proof}

We now prove that the limit infimum and limit supremum of $\lbrace \inlinenorm{ \dot F(\theta_k) }_2 \rbrace$ cannot be distinct. 

\begin{lemma} \label{lemma-limsup-grad-unconstrained}
Let $\lbrace \theta_k \rbrace$ be defined as in \eqref{eqn:iterates} satisfying \cref{P1,P2,P4}. 
Suppose \cref{A1,A2,A3b,A4b} hold. 
Then, for all $\delta > 0$,
\begin{equation}
\condPrb{  \norm{ \dot{F}(\theta_{k+1})}_2 > \delta, \norm{ \dot{F}(\theta_k)}_2 \leq \delta , ~i.o.}{\mathcal{F}_0} = 0,
\end{equation}
with probability one.
\end{lemma}
\begin{proof}
Let $\epsilon > 0$. For $\delta > 0$,
\begin{align}
&\condPrb{ \norm{ \dot{F}(\theta_{k+1})}_2 \Ind{ \norm{ \dot{F}(\theta_k)}_2  \leq \delta } > \delta + L \epsilon^\alpha }{\mathcal{F}_0} \\
&= \mathbb{P} \bigg{[} \left( \norm{ \dot{F}(\theta_{k+1})}_2 - \norm{\dot{F}(\theta_k)}_2 + \norm{\dot F(\theta_k)}_2 \right) 
\Ind{ \norm{ \dot{F}(\theta_k)}_2  \leq \delta } > \delta + L \epsilon^\alpha \bigg{\vert}  \mathcal{F}_0 \bigg{]} \\
&\leq \mathbb{P} \bigg{[} L \norm{ \theta_{k+1} - \theta_k}_2^\alpha \Ind{ \norm{ \dot{F}(\theta_k)}_2 \leq \delta } > L \epsilon^\alpha \bigg{\vert} \mathcal{F}_0 \bigg{]} \\
&= \mathbb{P} \bigg{[} \norm{ M_k \dot f(\theta_k, X_{k+1} )}_2 \Ind{ \norm{ \dot{F}(\theta_k)}_2 \leq \delta } > \epsilon\bigg{\vert} \mathcal{F}_0 \bigg{]} \\
&\leq \frac{1}{\epsilon^2} \norm{M_k}_2^2 \condE{ \norm{ \dot f(\theta_k, X_{k+1}) }_2^2 \Ind{ \norm{ \dot{F}(\theta_k)}_2 \leq \delta } }{\mathcal{F}_0} \\
&\leq \frac{1}{\epsilon^2} \norm{M_k}_2^2 \condE{ C_1 + C_2 ( F(\theta_k) - \Flb ) + (C_3-1) \delta^2 }{\mathcal{F}_0},
\end{align}
where we make use of \cref{A3b} in the last line. Moreover, by \cref{corollary-convergence-expected-objective-unconstrained}, we conclude
\begin{equation}
\begin{aligned}
&\condPrb{ \norm{ \dot{F}(\theta_{k+1})}_2 \Ind{ \norm{ \dot{F}(\theta_k)}_2  \leq \delta } > \delta + L \epsilon^\alpha }{\mathcal{F}_0} \\
&\quad \leq \frac{1}{\epsilon^2} \norm{M_k}_2^2 \left( C_1 + C_2 M_0 + (C_3-1) \delta^2 \right),
\end{aligned}
\end{equation}
where $M_0 = \sup_k \inlinecondE{ F(\theta_k) - \Flb}{\mathcal{F}_0}$ is finite.

By \cref{P2}, the sum of the last expression over all $k+1 \in \mathbb{N}$ is finite. By the Borel-Cantelli lemma, for all $\delta > 0$ and $\epsilon > 0$,
\begin{equation}
\condPrb{  \norm{ \dot{F}(\theta_{k+1})}_2 > \delta + L \epsilon^\alpha, \norm{ \dot{F}(\theta_k)}_2 \leq \delta , ~i.o.}{\mathcal{F}_0} = 0,
\end{equation}
with probability one.
Since this holds for any $\epsilon > 0$, it will hold for every value in a sequence $\epsilon_n \downarrow 0$. Since the countable union of measure zero events has measure zero,
\begin{equation}
\condPrb{  \norm{ \dot{F}(\theta_{k+1})}_2  > \delta, \norm{ \dot{F}(\theta_k)}_2 \leq \delta , ~i.o.}{\mathcal{F}_0} = 0,
\end{equation}
with probability one.
\end{proof}

We now put the two preceding lemmas together to prove the result.
\begin{corollary} \label{corollary-convergence-grad-unconstrained}
Let $\lbrace \theta_k \rbrace$ be defined as in \eqref{eqn:iterates} satisfying \cref{P1,P2,P3,P4}. 
Suppose \cref{A1,A2,A3b,A4b} hold. 
Then, $\lim_{k \to \infty} \inlinenorm{ \dot{F}(\theta_k)}_2 = 0$ with probability one.
\end{corollary}
\begin{proof}
For any $\delta > 0$, \cref{lemma-liminf-grad-unconstrained} implies
\begin{equation}
\begin{aligned}
&\condPrb{ \norm{\dot{F}(\theta_{k+1})}_2 > \delta, ~i.o. }{\mathcal{F}_0} \\
&\quad = \condPrb{ \left\lbrace \norm{\dot{F}(\theta_{k+1})}_2 > \delta \right\rbrace \cap \left\lbrace \norm{\dot{F}(\theta_{k})}_2 \leq \delta, ~i.o. \right\rbrace }{\mathcal{F}_0},
\end{aligned}
\end{equation}
with probability one. We see that this latter event is exactly, 
\begin{equation}
\condPrb{  \norm{ \dot{F}(\theta_{k+1})}_2  > \delta, \norm{ \dot{F}(\theta_k)}_2 \leq \delta , ~i.o.}{\mathcal{F}_0},
\end{equation}
which, by \cref{lemma-limsup-grad-unconstrained}, is zero with probability one. Therefore, $\condPrb{ \inlinenorm{\dot{F}(\theta_{k+1})}_2  > \delta, ~i.o. }{\mathcal{F}_0}$ is zero with probability one. Letting $\delta_n \downarrow 0$ and noting that the countable union of measure zero sets has measure zero, we conclude $\inlinecondPrb{ \inlinenorm{\dot F(\theta_{k})}_2 > 0, ~i.o. }{\mathcal{F}_0} = 0$ with probability one. In other words, $\inlinecondPrb{ \lim_{k \to \infty} \inlinenorm{ \dot F(\theta_{k}) }_2 = 0 }{\mathcal{F}_0} = 1$ with probability one.
\end{proof}

\subsection{Asymptotic Behavior of the Expected Gradient Function} \label{subsection-global-holder-expected-gradient}

We begin by proving that $\sup_k \inlinecondE{ \inlinenorm{ \dot{F}(\theta_k)}_2^2 }{\mathcal{F}_0}$ is finite with probability one. As a result, we will have that $\lbrace \dot{F}(\theta_k) \rbrace$ are uniformly integrable, which, with \cref{corollary-convergence-grad-unconstrained}, implies $L^1$ convergence.

\begin{lemma} \label{lemma-grad-bounded-objective}
Suppose \cref{A1,A4b} hold. Then, for all $\phi \in \mathbb{R}^p$,
	\begin{align}
	\norm{ \dot{F}(\phi)}_2^2 \leq \left( \frac{L^{\frac{1}{\alpha}}(1+\alpha)}{\alpha} [F(\phi) - \Flb] \right) ^{\frac{2\alpha}{1+\alpha}},
	\end{align}
where $2\alpha/(1+ \alpha) \leq 1$ for all $\alpha \in (0,1]$.

Moreover, let $\lbrace \theta_k \rbrace$ be defined as in \eqref{eqn:iterates} satisfying \cref{P1,P2,P4}. 
Suppose \cref{A1,A2,A3b,A4b} hold. Then, $\sup_k \inlinecondE{ \inlinenorm{ \dot F (\theta_k) }_2^2}{\mathcal{F}_0} < \infty$ with probability one.
\end{lemma}
\begin{proof}
By \cref{lemma-ftc} and \cref{A1}, for any $\phi,\theta \in \mathbb{R}^p$,
\begin{equation}
0 \leq F(\phi) - \Flb + \dot{F}(\phi)'(\theta - \phi) + \frac{L}{1+\alpha} \norm{ \theta - \phi}_2^{1+\alpha}.
\end{equation}

We now find the $\theta$ that minimizes this upper bound, and plug it back into the upper bound. By rearranging, we conclude the result. 

For the second part, by \cref{corollary-convergence-expected-objective-unconstrained}, $M_0 := \sup_k \inlinecondE{ F(\theta_k) - \Flb}{\mathcal{F}_0} < \infty$ with probability one. By plugging $\theta_k$ into the first part of the result, taking expectations and applying H\"{o}lder's inequality,
\begin{equation}
\condE{ \norm{ \dot F(\theta_k)}_2^2}{\mathcal{F}_0} \leq \left( \frac{L^{\frac{1}{\alpha}}(1+\alpha)}{\alpha} M_0 \right)^{\frac{2\alpha}{1+\alpha}},
\end{equation}
with probability one. The result follows. 
\end{proof}

\begin{corollary} \label{corollary-convergence-expected-grad-unconstrained}
Let $\lbrace \theta_k \rbrace$ be defined as in \eqref{eqn:iterates} satisfying \cref{P1,P2,P3,P4}. 
Suppose \cref{A1,A2,A3b,A4b} hold. 
Then, $\lim_{k \to \infty} \inlinecondE{\inlinenorm{ \dot{F}(\theta_k)}_2}{\mathcal{F}_0} = 0$ with probability one.
\end{corollary}
\begin{proof}
By \cref{lemma-grad-bounded-objective}, $\lbrace \inlinenorm{\dot F(\theta_k)}_2 \rbrace$ are bounded in $L^2$. Therefore, the sequence is uniformly integrable. In light of the uniform integrability of the sequence and \cref{corollary-convergence-grad-unconstrained}, we can conclude the result.
\end{proof}

%\section{Analysis of Radial Growth Conditions} \label{section-analysis-radial-growth}
%\input{sections/analysis_radial_growth}

\end{document}